\documentclass[11pt]{amsart}

\usepackage[shortalphabetic]{amsrefs}
\usepackage{eucal,graphicx,amssymb,mathrsfs}
\usepackage[all, knot]{xy}

\theoremstyle{definition}
\newtheorem{ex}{Example}[section]
\newtheorem*{ex*}{Example}

\newtheorem{question}{Question}

\theoremstyle{plain}
\newtheorem{thm}{Theorem}[section]
\newtheorem{prop}[thm]{Proposition}
\newtheorem*{cor*}{Corollary}

\newtheorem*{claim}{Claim}
\newtheorem{thm*}{Theorem}

\newtheorem*{thmA}{Theorem A}
\newtheorem*{thmA'}{Theorem A'}
\newtheorem*{thmB'}{Theorem B'}
\newtheorem*{thmB}{Theorem B}

\theoremstyle{remark}
\newtheorem*{rem}{Remark}

\numberwithin{equation}{section}

\DeclareMathOperator{\red}{red}
\DeclareMathOperator{\id}{id}

\DeclareMathOperator{\codim}{codim}

\DeclareMathOperator{\Hom}{Hom}

\DeclareMathOperator{\PGL}{PGL}
\DeclareMathOperator{\GL}{GL}
\DeclareMathOperator{\diag}{diag}

\def\C{\mathbb{C}}

\def\P{\mathbb{P}}

\def\Z{\mathbb{Z}}
\def\P{\mathbb{P}}

\def\O{{\mathcal{O}}}

\def\cA{{\mathcal{A}}}

\def\frm{{\mathfrak{m}}}
\def\d{{\delta}}

\def\n0{{\bf n_0}}

\def\la{{\lambda}}

\def\dto{\dashrightarrow}

\begin{document}

\title[Rational Maps Induced from Algebraic Structures]
{Degree Growth of Rational Maps Induced from Algebraic Structures}

\author{Charles Favre}
\address{CMLS, \'Ecole polytechnique, CNRS, Universit\'e Paris-Saclay, 91128 Palaiseau Cedex, France}
\email{charles.favre@polytechnique.edu}
\author{Jan-Li  Lin}
\address{Department of Mathematics\\
	Northwestern University\\
	Evanston, IL 60208}
\email{janlin@math.northwestern.edu}

\thanks{First author is supported by the ERC-starting grant project "Nonarcomp" no.307856.}

\begin{abstract}
For a finite dimensional vector space equipped with a $\mathbb C$-algebra structure, one can define rational maps using the algebraic structure.  In this paper, we describe the growth of the degree sequences for this type of rational maps.
\end{abstract}

\date{\today}

\maketitle
\tableofcontents

\section{Introduction}

In higher dimensional complex dynamics, understanding the degree growth of a rational map under iteration is a fundamental and important
issue. Given any dominant rational map $f:\P^d\dashrightarrow\P^d$, one can define its $p$-th degree
$\deg_p(f):=\deg(f^{-1} L_p)$, where $L_p$ denotes a generic linear subspace of $\P^d$ of codimension $p$. 
The problem is to describe  the behavior of the sequence $\{\deg_p(f^n)\}_{n=1}^\infty$, especially as $n\to\infty$. It is not difficult to check that this sequence is sub-multiplicative. Following~\cite{MR1488341} one can therefore introduce a numerical invariant which measures the exponential growth of the degree
\[
\lambda_p(f):=\lim_{n\to\infty} \deg_p(f^n)^{1/n} \ge 1.
\]
These are invariants of birational conjugacy that are usually referred to as  the {\em dynamical degrees}. 

Understanding the degrees of iterates or computing the dynamical degrees of a rational selfmap are not easy tasks.
Most results have been focused on the case $p = 1$, 
see \cite{DF,FJ,AABM,AMV,BK1,BK2,BHM,Ng} and the references therein.
On the other hand, the case $2 \le p \le d-2$ is substantially harder since it is 
delicate to compute  $\deg_p(f)$ even in concrete examples, and
there are only a few references in the
literature.
For monomial maps, the degree growth is obtained in \cite{FW, L}. Other family of maps
have also been investigated, see \cite{Og,DN}. 

\smallskip

In this paper, we study the behavior and the degree growth of rational maps
which come from an algebra structure. 
One motivation to study such maps comes from the recent 
 classification of birational maps of type $(2,2)$
on $\P^d$ in terms of involutions on certain Jordan algebra by L. Pirio and F. Russo~\cite{PiRu1, PiRu2,PiRu3}.  
There are also some sporadic studies in the literature about the dynamics of this type of maps,
for example, see \cite{CeDe, Usn} which study the dynamics of ratioanl maps on the matrix algebras. In this paper we study the algebraic structure of these maps more systematically and describe the growth of their degrees.

\smallskip

Our starting point is a finite dimensional complex vector space $V$
equipped with a $\C$-algebra structure. 
We shall be mainly interested in the dynamics of two classes of maps. 

%

The first family of maps we deal with are the maps induced from single variable rational map.
Here we need to assume that the algebra is {\em power associative} and has a {\em multiplicative unit} (see Section~\ref{sec:setting}
for definitions). This implies that the power map $x\mapsto x^n$, $n\ge 0$ is well-defined, where $x^n$ is the multiplication in $V$ of $x$
with itself for $n$ times in any order. As a consequence, for a single variable polynomial $P(T)\in\C[T]$, $P$ also induces a self map on $V$.
Furthermore, one can show that a generic element of $V$ is invertible. Therefore,  any rational function
$\varphi=\frac{Q}{P}\in\C(T)$, $\varphi$ induces a rational selfmap $f_\varphi:V\dashrightarrow V$. Compactifying $V$, we obtain
a rational map $f_\varphi:\P^d\dashrightarrow \P^d$ where $d=\dim_\C(V)$.

\begin{thmA}
Suppose $V$ is a power associative $\C$-algebra with a multiplicative unit.
Then there exists an integer $k\ge 1$ such that 
for any rational map
$\varphi(T)=\frac {Q(T)}{P(T)}\in \C(T)$ and for  any integer $p \ge1$, we have
$$
\deg_p (f_\varphi^n)  \asymp \left( \max\{ \deg(Q), \deg(P)\}^{\min \{ p, k \}}\right)^n.
$$
\end{thmA}


The second type of maps we study are the (generalized) monomial maps. In order for the iterates to be still monomial maps,
we need to require that the algebra structure is {\em abelian} (i.e., commutative, associative and unitary). Now we let $k=\dim_\C(V)$, given an $d\times d$ integer matrix
$A=(a_{ij})\in M_d(\Z)$ with $\det(A)\ne 0$, we can define the monomial map
$F_A:V^d \dashrightarrow V^d$ as
\[
\textstyle{ F_A(x_1,\cdots,x_d)=\left(\prod_j x_j^{a_{1,j}},\cdots ,\prod_j x_j^{a_{d,j}} \right),  }
\]
where $x_j\in V$ and the product is given by the multiplication of $V$.
Let the number $m$ be given by
$m=\dim_\C(\red(V))$, where $\red(V)=V/N(V)$
and $N(V)$ is the nilradical of $V$, i.e., the ideal of all nilpotent elements.

\begin{thmB}
\label{thm:B}
Let $\diag(A;m)$ be the block-diagonal matrix with $m$ blocks of the matrix $A$ on the diagonal positions. Then
the degree growth of the (generalized) monomial map $F_A$ is given by
\[
\deg_p(F_A^n) \asymp \max_{p-d(k-m)\le i\le p} \left\| \wedge^{i} \diag(A;m)^n \right\|.
\]
\end{thmB}


Observe that $m$ only depends on the
structure of $V$ and not on the map $F_A$. 
Also, the norm $\|\wedge^i \diag(A,m)\|$ can be computed solely in terms of $\|\wedge^j A^n\|$, $j\le i$, and $m$. Indeed, one has
\[
\| \wedge^i \diag(A,m)^n \| \asymp \max_{i_1 + ... + i_m = i} \| \wedge^{i_1} A^n \| \cdots \| \wedge^{i_m} A^n \|.
\]

\smallskip

There are many  other possibilities to associate a rational map to a structure of algebra. We point out here some generalizations of our setting and natural  questions that  may arise.

\begin{question}
Let $V$ be an abelian algebra, and pick any rational map
\[
f=[f_0:\cdots:f_d]:\P^d_\C\dashrightarrow\P^d_\C,
\]
where each $f_j$ is a homogeneous polynomial of the same degree. Then $f_j$ also induces a map $V^{d+1}\to V$ which is homogeneous of the same degree.
Thus $f$ induces a map $F:\P(V^{d+1})\dashrightarrow\P(V^{d+1})$. Can we describe the relation between the degree growth
of $f^N$ and the degree growth of $F^N$?
\end{question}

\begin{question}
Under the same notation as the above question, one can show that $F$ is birational when 
$f$ is birational. In the particular case where $d=2$, it is tempting to explore if 
similar results as in \cite{DF} can be obtained about the degree growth of  the iterates of $F$?
\end{question}

\begin{question}
If $V$ is a power associative algebra with a multiplicative unit, and $P(T)\in\C[T]$ is a polynomial,
then we can define the generalized H\'{e}non map $H_P: V^2\to V^2$ by
\[
H_P(x,y)=(y, P(x)-cy)
\]
where $c\in\C\setminus\{0\}$. What is the degree growth of the iterates of $H_P$?
\end{question}

\section{Algebra structure and quadratic maps}

\subsection{The algebra structure}
\label{sec:setting}

Let $V$ be a finite dimensional complex vector space
equipped with a $\C$-algebra structure, that is a $\C$-linear map $\mu:V\otimes V\to V$. We shall always denote multiplicatively this law, i.e. we will denote $\mu(x\otimes y)$ by $xy$.

Recall the following list of 
classical  definitions.
\begin{itemize}
\item
$V$ is {\em unitary} if it has a unit $1$ (i.e. $1 \cdot x = x \cdot 1 = x$ for all $x$).
\item
$V$ is {\em associative} if $x (yz) = (xy) z$ for all $x,y,z$;
\item
$V$ is {\em commutative} if $x y = y x$ for all $x,y$;
\item
$V$ is {\em alternative} if $x(xy) = x^2 y$ and $(yx)x = y x^2$ for all $x,y$;
\item
$V$ is {\em power-associative} if the algebra generated by any element is associative;
\item
$V$ is {\em abelian} if the multiplication law is commutative, associative and unitary
(this is the typical setting for commutative algebra, e.g., \cite{AM});
\item
$V$ is a {\em Jordan algebra} if it is commutative and alternative.
\end{itemize}

Set $ d = \dim_\C V$. Observe that the space $\cA$ of all algebra structures on $V$ is $\Hom_\C(V\otimes V, V)$, which is
an affine space of dimension $d^3$. 

The space of  abelian (resp. associative, alternative, power-associative, Jordan) algebras
is a Zariski closed subset of $\cA$.

When $V$ is power associative, we define the set of nilpotent elements as
$N(V) := \{ x\in V\,|\, x^N = 0 \text{ for some integer } N \}$. When $V$ is abelian then $N(V)$
is an ideal.


%
%
%
%
%

\subsection{Quadratic maps}

There is a natural identification between $d$ dimensional complex commutative (but not necessarily associative)
algebras and homogeneous quadratic polynomial maps from $\C^d$ to itself, as follows.

Choose a basis of $V$ as a complex vector space $V = \C e_1 \oplus \ldots \oplus \C e_d$, and write
$e_i\cdot e_j = \sum_k a_{ij}^k e_k$ with $a_{ij}^k\in \C$. Since $V$ is commutative, we have $a_{ij}^k=a_{ji}^k$. Then $f_V(x) = x^2$
is a  polynomial map as above, and we have
\[
f_V ( z_1, \ldots , z_d) = \left( \sum_{1\le i, j\le d}
a_{ij}^1 z_i z_j, \ldots , 
\sum_{1\le i, j\le d} a_{ij}^d z_i z_j
\right).
\]

Conversely, given a homogeneous quadratic polynomial map $f:\C^d\to\C^d$, one can define the algebra structure as
\[
x\cdot y = \frac 1 2 \left( f(x+y) - f(x) - f(y) \right).
\]

The two operations are inverse to each other. Therefore, if we use $f_V$ to denote again the induced quadratic map
$\P^{d-1}\dashrightarrow\P^{d-1}$, then we can see that the space of all quadratic rational selfmaps
on $\P^{d-1}$ can be identified with the space of $f_V$ over all algebra structures on $\C^d$.
Moreover, the following results show that the dynamics of $f_V$ plays a special role in the structure of $V$.

\begin{prop}
The algebra is unitary iff $f_V$ admits a fixed point $x$ such that $df_V (x) = 2 \id$.
\end{prop}
\begin{proof}
In one direction this is obvious since $ (1 + tx ) ^2  = 1 + 2tx + O(t^2)$.
In the other direction,  pick a fixed point $x$ and let $y$ be any other point.
Then $(x+ ty)^2 = x^2 + 2t (x\cdot y)  + O(t^2) = x^2 + df_V(x) (t y) + O(t^2)$,
whence $x \cdot y = y$ as required.
\end{proof}

\begin{prop}
Suppose $V$ and $V'$ are dimension $d$  commutative algebras.
Then $V$ is isomorphic to $V'$ iff $f_V : \P^{d-1} \dashrightarrow \P^{d-1}$ is conjugated in
$\PGL(d, \C)$ to $f_{V'}$.
\end{prop}
\begin{proof}
Identify both $V$ and $V'$ to $\C^d$.
Being isomorphic then means that there exists
$A \in \GL(d,\C)$ such that $ A ( x \cdot_V y ) = Ax \cdot_{V'} Ay$.
This implies the conjugacy in $\C^d$.

Conversely, being conjugated in $\P^{d-1}$ implies the existence of a matrix
in $\GL(d,\C)$ and a function $\la(x)$ such that
 $ A ( x \cdot_V x ) =  \lambda(x) \,  Ax \cdot_{V'} Ax$.
For degree reasons, $\la$ is a constant. Changing $A$ by $\mu A$
replaces $\la$ by $\la \mu$ hence we can assume $\la = 1$.
\end{proof}

\begin{rem}
Any two-dimensional unitary commutative algebra is abelian, hence isomorphic to either
$\C[x]/(x^2)$ or to $\C  \oplus \C$.

Indeed, the algebra structure is determined by $x^2 = a + bx$.
When $a + b^2/4  =  0$, then there exists an element $y:= \mu x + \la$ such that
$y^2 = 0$ and we are in the former case.
When $a + b^2/4 \neq 0$, then there exists an element $y:= \mu x + \la$ such that
$y^2 = y$ and we are in the latter case.
\end{rem}


\section{Some lemmas about degree growth}

In this section, we will prove several lemmas for the degrees and degree growth of rational maps.
The maps we study in this paper all preserve some fibration. The dynamical degrees for maps preserving
a fibration were computed by Dinh, Nguy{\^e}n and Truong \cite{DN,DNT} (see \cite{Tru} and \cite{NB} for a purely algebraic approach to these computations). Their results are very general and their computations are quite involved. For the convenience of the reader we have preferred to reprove some of their (elementary) results that we shall need in the latter section. 
%

First, let us define higher degrees. For a rational selfmap $f:X\dashrightarrow X$ on a projective variety,
and given an ample divisor $D$, we define the {\em $p$-th degree} of $f$ with respect to $D$ as
$\deg_{D,p}(f)=f^*[D]^p\cdot[D]^{d-p}$, where $[D]$ is the class of $D$ in $H^{1,1}(X)$. When $X=\P^d$ and $D=H$
is a generic hyperplane, this coincides with the usual definition of degrees. In this case, we can compute the degree
by $\deg_{p}(f)=f^{-1}L_p\, . \, L_{d-p}$, where $L_p$ and $L_{d-p}$ are generic linear subspaces of $\P^d$ of
codimensions $p$ and $d-p$ respectively, and $f^{-1}L_p$ is the proper transform of $L_p$ under $f$.

\subsection{Product map}
Product maps are the simplest  maps preserving a fibration. 
Let $f:X\dashrightarrow X$
and $g:X'\dashrightarrow X'$ be two rational maps, and $h=f\times g$. Pick any two ample divisors $D, D'$ on $X, X'$ respectively, and write $n= \dim X$, $n'= \dim X'$. By abuse of notation we shall again denote by $D$ and $D'$ their pull-back to $X\times X'$ so that $D+D'$ is again ample on $X\times X'$.  The degrees of $(f,g)$ can be then computed by a direct calculation as
\[
\deg_{D+D',p}(h)=\sum_{i+j=p}{p \choose i}\, {n+n'-p \choose n -i} \, \deg_{D,i}(f)\deg_{D',j}(g).
\]

In dynamics, we are interested in the growth of degrees under iteration. For this reason, we introduce the following notation
for {\em asymptotic equivalence}. For two sequences $\{a_j\}$, $\{b_j\}$ of positive numbers, we say the two sequences are
{\em asymptotic equivalent}, denoted by $a_j \asymp b_j$, if there is a constant $C>0$ such that $C^{-1}a_j\le b_j\le C a_j$
for all large enough $j$.

 For a rational map $f$, the asymptotic behavior of the degree sequence $\{\deg_{D,p}(f^n)\}_{n=1}^\infty$
is independent of the ample divisor $D$, and is invariant under birational conjugation
(see \cite{DS} using analytic arguments or \cite{Tru,NB} using algebraic ones).
An important class of invariants to measure the asymptotic growth of the degree sequence is the {\em dynamical degrees}.
The $p$-th dynamical degree of a rational map $f$, denoted by $\lambda_p(f)$, is defined as
\[
\lambda_p(f) = \lim_{n\to\infty} \deg_{D,p}(f^n)^{1/n}.
\]
The existence of the limit follows here from the sub-multiplicativity of the sequence $\{C\deg_{D,p}(f^n)\}_{n=1}^\infty$
for some constant $C>0$, see~\cite{DS}.

With the above notation, we can first conclude that for product map $h=f\times g$, we have
\begin{equation}\label{eq:product}
\deg_{D+D',p}(h^n)\asymp\max_{i+j=p}\left\{\deg_{D,i}(f^n)\deg_{D',j}(g^n)\right\}.
\end{equation}
A consequence of this is a formula for the dynamical degrees of $h$:
\[
\la_p(h)=\max_{i+j=p}\left\{\la_i(f)\la_j(g)\right\}.
\]

\subsection{Skew product with fibers $\P^1$}
Another type of rational maps preserving a fibration that we will encounter
is a map $F:\P^d\times\P^1\dashrightarrow\P^d\times\P^1$
which preserves the fibration $\pi:\P^d\times\P^1\to\P^d$.
Denote by $H_d$ the pull-back of $\O(1)_{\P^d}$ on  $\P^d\times\P^1$, and by $H_1$
 the pull-back of $\O(1)_{\P^1}$ on the same variety.

\begin{prop}
\label{lem:deg_growth}
Suppose that the rational map $F:\P^d\times\P^1\dashrightarrow\P^d\times\P^1$ is of the form
$F(z,t)=(g(z),h(z,t))$, where $z\in\P^d$, $t\in\C\cup\{\infty\}=\P^1$. 
Suppose $F^* H_1 = \delta_d H_d + \delta_1 H_1$. 

The $p$-th degree of $F$ with respect to the ample divisor $D:= H_d + H_1$ can be computed as
\[
\deg_{D,p}(F)= (d+1-p) \deg_p(g) +p  (\delta_1+(d+1-p)\delta_d)\cdot\deg_{p-1}(g).
\]
\end{prop}
Observe that $\delta_1$ is the degree in $t$ of the rational function $h(z,t)$ whereas $\delta_d$
is the degree in $z$ of $h(z,t)$.
\begin{proof}[Proof of Proposition~\ref{lem:deg_growth}]
One has 
\begin{align*}
F^* H_d ^p &= \deg_p(g) H_d^p,\\ 
F^* (H_d ^{p-1} \cdot H_1) &= \deg_{p-1}(g) H_d^{p-1} \cdot ( \delta_d H_d + \delta_1 H_1).
\end{align*}
The first equality is by definition. For the second equality, we first claim the following: 

\begin{claim}
for any subvariety $Z$ of $\P^d\times\P^1$, we can find a linear subpace $L\subset \P^d$ of codimension $p-1$ and a point $q\in \P^1$ such that $L\times \{q\}$, $L\times\P^1$ and $\P^d\times\{q\}$ all intersect $Z$ properly.
\end{claim}

Recall that two pure dimensional subvarieties $W$ and $W'$ intersect properly when the codimension of any irreducible component of $W \cap W'$ is equal to $\codim(W) + \codim(W')$. 

\smallskip

Here we show the claim for $L\times \{q\}$, and leave the other two cases to the reader. Since $Z$ is irreducible, the projection of $Z$ on $\P^1$ is either a single point $\{q'\}$ or $\P^1$. In the first case, we can pick any point $q\ne q'$ in $\P^1$ and any linear subvariety $L$ of codimension $p-1$. For the latter case, pick a generic $q\in\P^1$ such that $Z' = (\P^d\times\{q\})\cap Z$ is a variety of pure dimension $\dim(Z)-1$. Next, pick a generic linear subspace $L\in\P^d$ which intersects $Z'$ properly. Then $L\times\{q\}$ will intersect $Z$ properly.

Now let $\Gamma\subset(\P^d\times\P^1)\times(\P^d\times\P^1)$ be the graph of $F$. Denote by $\pi_1$ and $\pi_2$ the projections from $\Gamma$ onto the first and second components, and  set
\[
Z=\{q\in \P^d\times\P^1\ |\ \dim(\pi_2^{-1}(q))>0 \}~.
\]

The class $\pi_2^*(H_d^{p-1} \cdot H_1)$ is represented by $\pi_2^{-1}(L\times\{p\})$ (see e.g. ~\cite[Lemma~3.1]{Tru}), which is equal to
$\pi_2^{-1}(L\times\P^1)\cap\pi_2^{-1}(\P^d\times\{q\}) $.
On the other hand this intersection represents the class $\pi_2^*H_d^{p-1}\cdot \pi_2^*H_1$. Thus we have 
\[
\pi_2^*(H_d^{p-1} \cdot H_1) = \pi_2^*H_d^{p-1}\cdot \pi_2^*H_1.
\]
Notice that $\pi_2^* H_d^{p-1} = \pi_1^*(\deg_{p-1}(g) H_d^{p-1})+E$ for some class $E$ such that $\pi_{1*}(E)=0$, therefore
\[
\begin{split}
F^* (H_d^{p-1} \cdot H_1) & =  \pi_{1*} \pi_2^*(H_d^{p-1} \cdot H_1)=  \pi_{1*} (\pi_2^*H_d^{p-1}\cdot \pi_2^*H_1)\\
& =  \pi_{1*} (\pi_1^*(\deg_{p-1}(g) H_d^{p-1})\cdot \pi_2^*H_1)+ \pi_{1*} (E\cdot \pi_2^*H_1)\\
& =   \deg_{p-1}(g) H_d^{p-1}\cdot (\pi_{1*}\pi_2^*H_1) +0\\
& =   \deg_{p-1}(g) H_d^{p-1}\cdot F^*H_1
=   \deg_{p-1}(g) H_d^{p-1}\cdot (\delta_d H_d + \delta_1 H_1).
\end{split}
\]
Finally, we can compute
\begin{multline*}
F^* (D^p) = F^* \left( 
H_d^p +p\,  H_d^{p-1} \cdot H_1
\right) = \\
\left( \deg_p(g) +p \deg_{p-1}(g) \delta_d\right) H_d^p +  p \deg_{p-1}(g) \delta_1\, H_d^{p-1} \cdot H_1
\end{multline*}
and the result follows by intersecting with the class 
$$D^{d+1-p} = H_d^{d+1-p} + (d+1-p) H_d^{d-p}\cdot H_1~.$$
\end{proof}

\section{Degree growth of the squaring map}

In this section we study  the degree growth of the squaring map $f_V:\P^{d-1}\dashrightarrow\P^{d-1}$
 introduced in the previous section.
The methods we use in this section will be generalized in  later sections to prove our main Theorems A and B.
For squaring maps, these methods are more intuitive and serve as an illustration for the more complicated cases.

\subsection{The abelian case}
\label{sec:square_abelian}
We assume $V$ is an abelian algebra.
Since $V$ is finite dimensional as a complex vector space, the $\C$-algebra $V$ is Artinian
(\cite{AM}*{$\S 8$, Exercise 2}).
By the structure theorem for Artin rings (\cite[Theorem 8.7]{AM}), we can decompose $V$ as
a finite direct product of Artinian local rings $V \simeq \prod_{i=1}^k V_i$.

For the Artin local ring $V_i$, the maximal ideal $\frm_i$ is nilpotent
(i.e. $\frm_i^l=0$ for some $l\ge 1$) and $V_i/\frm_i\simeq \C$.
As a $\C$-vector space, we can write $V_i$ as the
direct sum $V_i\simeq \C\oplus \frm_i$.
For $V$, we introduce the map
\[
\Phi : V \simeq \C^k \times \prod_{i=1}^k \frm_i \longrightarrow V
\]
sending $((a_1, \ldots, a_k), (h_1, \ldots, h_k))$ to
$\sum_{i=1}^k a_i \exp(h_i)$, where $$\exp(h)=\sum_{j=0}^\infty h^j/j!~.$$
Notice that $\exp(h)$ is well defined since $h_i\in\frm_i$ is nilpotent,
so the sum is indeed finite.

Moreover, we claim that $\Phi$ is a birational map.
Indeed, using the vector space isomorphism $V \simeq \prod_{i=1}^k \C \times \frm_i$,
we can describe the birational inverse of $\Phi$ concretely as
\[
\Phi^{-1}\left(\prod_{i=1}^k (a_i,h_i)\right)= \prod_{i=1}^k a_i\cdot \log(1+a_i^{-1}h_i)~,
\]
where  $\log(1+x)=\sum_{j=1}^\infty (-1)^{j-1}x^j/j$.

The usual rules for the exponential and logarithm functions hold. Hence the fact that
$\Phi$ and $\Phi^{-1}$ are inverse to each other is a reflection of the fact that $\log(\exp(h))=h$.
However, we emphasize here that $\Phi$ and $\Phi^{-1}$ are {\em not} ring homomorphisms.
They are inverse to
each other only as rational maps.

%
%
%
%
%

Define $$F\left[(a_1, \ldots, a_k), (h_1, \ldots, h_k) \right]=\left[ (a_1^2, \ldots, a_k^2), (2h_1, \ldots, 2h_k)\right]$$ then
$\Phi \circ F =  f_V \circ \Phi$.
Thus, $f_V$ is birationally conjugate to a product of power maps and linear maps.

Recall that the reduced algebra associated to $V$ is by definition the quotient of $V$ by its nilradical  $N(V)$.
Since for each $V_i$, we have $\red(V_i)\cong\C$, thus $\red(V)\cong\C^k$.
From the product structure of $f_V$, we 
obtain the following.
\begin{thm}
Suppose $V$ is abelian,  and write $k := \dim_\C \red (V)$. Then for any $p$, we have
$$
\deg_p (f_V^n)  \asymp \left(2^{\min \{ p, k \}}\right)^n
$$
\end{thm}
\mbox{}\hfill\qed

\begin{ex}
Suppose that $V$ is power associative and generated by one element. Then
it is automatically abelian, and there exists a polynomial $P(x)\in\C[x]$ such that
$$
V \simeq \C[x] / (P(x))\simeq \C[x] / \prod_{i=1}^k (x - z_i)^{k_i}
\simeq \oplus_{i=1}^k \C [x_i] /(( x_i-z_i)^{k_i}).
$$
In this case, $k$ is the number of different (complex) roots of the defining polynomial $P$.
\end{ex}



\subsection{Power associative algebras}\label{sec:pwassoc}

In this section, we will deal with the squaring map $f_V$ when $V$ is a power associative algebra with $1$.
%
%
We start by analyzing the structure of the algebra.
For any non-zero $x\in V$ denote by $\C[x]$ the algebra generated by $x$ and
$\d(x) := \dim_\C \C[x]$.
Since $\C$ is always a subspace of $\C[x]$, $\d(x)\ge 1$, and
since $V$ is power associative, $\C[x]$ is abelian.
Moreover, $\C[x]$ is invariant under $f_V$, i.e., $f_V(\C[x])\subseteq\C[x]$.

Observe that
\[
V_k = \{ x\in V\, | \, (1, x, x^2, \ldots, x^k) \text{ are linearly dependent}\}
\]
is a Zariski closed subset of $V$ since it is defined by the vanishing
of finitely many determinants of matrices of size $(k+1)\times(k+1)$. Since $V_k = \{x\,|\, \d(x) \le k \}$
and $V_k\subseteq V_l$ if $k\le l$,
we conclude that the function $x\mapsto \d(x)$ is lower semicontinuous for the Zariski topology.

Introduce $\d=\d_V := \max \{\d (x)\,|\, x\in V\}$ and $U':= \{ x \in V, \, \d(x) = \d_V\}$.
The latter is a Zariski dense open subset of $V$.
Let $F := V\setminus U'$ and pick $x, y\in U'$,
we have the following observations:
\begin{enumerate}
\item $y\in\C[x]\ \Longleftrightarrow\ x\in\C[y]\ \Longleftrightarrow\ \C[x]=\C[y]$,
\item if $y \notin \C[x]$, then $\C[x]\cap \C[y]\cap U' = \emptyset$, i.e. $\C[x]\cap \C[y]\subseteq F$.
\end{enumerate}

Moreover, we claim the following:

\begin{claim}
There is a further open dense subset $U\subset U'$ such that for any two $x,y\in U$, we have $\C[x]\cong\C[y]$ as $\C$-algebras.
\end{claim}

\begin{proof}[Proof of the Claim]
For each $x\in U'$, there is a canonical map $\C[T]\to\C[x]$, where $T$ is a variable, by sending
$T\mapsto x$. Thus, $\C[x]\cong\C[T]/(P_x(T))$ for a unique monic polynomial $P_x(T)$ of degree $\d$.
The coefficients of $P_x(T)$ depend algebraically on $x$.
Let $S^\d\C$ denote the $\d$-th symmetric product of $\C$.
Then one can parameterize all monic polynomials of degree $\d$ either by $\C^\d$ (using the coefficients, excluding the leading one),
or by $S^\d \C$ (using the $\d$ roots of the polynomial, counting multiplicities); and the two spaces $S^\d\C$ and $\C^\d$ are isomorphic.

We define the map $U\to S^\d\C$ by sending $x$ to the multiset of complex solutions of $P_x(T)$.
Each multiset consists of $\d$ elements, hence gives rise to a partition of $\d$ by counting the
multiplicities of different elements.
Denote $\Gamma$ as the set of all partitions of $\d$, then the above process defines a map $S^\d\C \to \Gamma$.
Use $\psi$ to denote the composition $U\to S^\d\C \to \Gamma$. Define a partial order $\preceq$ on $\Gamma$ by
$\gamma\preceq\gamma'$ if $\gamma'$ is a refinement of $\gamma$.
Also, define $U_\gamma = \{ x\in U'\ |\ \psi(x)\preceq \gamma\}$.

Notice that each inequality $\gamma\prec\gamma'$ (this notation means $\gamma\preceq\gamma'$ but $\gamma\ne\gamma'$)
can be factored into a sequence of minimal refinements
such that each partition in the sequence is obtained by decomposing one number in the
previous partition as the sum of two positive numbers. For instance, we may have
$\gamma=\{\nu_1,\nu_2\}$ and $\gamma'=\{\nu'_1,\nu''_1,\nu_2\}$ where $\nu'_1+\nu''_1=\nu_1$
as an example of a minimal refinement. Then an element $x\in U_{\gamma'}$ will have the corresponding
polynomial $P_x(T)$ of the form $P_x(T)=(T-\alpha_0)^{\nu'_1}(T-\alpha_1)^{\nu''_1}(T-\alpha_2)^{\nu_2}$,
where the $\alpha_i$'s are not necessarily different. And $x\in U_\gamma\subset U_{\gamma'}$
will then defined by the close condition $\alpha_0=\alpha_1$ in $U_{\gamma'}$.
Generalizing the above instance, one can obtain the conclusion that if $\gamma\prec\gamma'$,
then $U_\gamma\subset U_{\gamma'}$ is a closed subset.
Therefore, with respect to this partial order, $\psi$
is lower semicontinuous.

Lower semicontinuity implies that we can find $\gamma_1,\cdots,\gamma_m\in\Gamma$, no two of them are comparable under ``$\preceq$'',
such that if we define $U_i = \{ x\in U'\ |\ \psi(x)\preceq \gamma_i\}$, then each $U_i$ is closed in $U'$, and $U'=\cup_{i=1}^m  U_i$.
However, since $U'$ is an open dense subset of the irreducible space $V$, $U'$ is itself irreducible.
This implies that $m=1$, and there
is a unique partition $\gamma\in\Gamma$, and an open dense subset $U\subset U'$ such that $U=\{x\in U'\ |\ \psi(x)=\gamma\}$.

Finally, since we have $\C[T]/((T-\alpha)^k)\cong\C[T]/(T^k)$ for all $\alpha\in\C$ (via the isomorphism $T\mapsto T+\alpha$),
the isomorphic class of $\C[x]$ for $x\in U$ only depend on the multiplicity of different roots of $P_x(T)$,
which is recorded as $\psi(x)$. Since
$\psi(x)=\gamma$ for all $x\in U$, we conclude that for all $x\in U$, $\C[x]$ are isomorphic to each other.
\end{proof}

Fix a monic polynomial $P_0(T)$ of degree $\d$ whose roots gives rise to the partition $\gamma$.
Then for each $x\in U$, we have the isomorphism
\[
\phi_x:V_0:=\C[T]/(P_0(T))\longrightarrow \C[x]\cong\C[T]/(P_x(T)).
\]
Notice that $\phi_x$ depend algebraically on $x$.

Next, choose a generic affine (i.e, a translation of a linear) subspace $L\subset V$ of dimension $d-\d$, such that
$L\cap U$ is open and dense in $L$. Then for a generic $x\in U$, $\C[x]$ intersects $L$
at a single point. One defines the following birational map
\[
\begin{split}
\Phi: L\times V_0 & \dashrightarrow  V   \\
         ( x, v ) & \longmapsto      \phi_x(v) \in \C[x]\subset V,
\end{split}
\]
whose inverse is given by
\[
\begin{split}
\Psi: V & \dashrightarrow L\times V_0\\
      y & \longmapsto \left(x=\C[y]\cap L\, ,\, \phi_x^{-1}(y)\right)~.
\end{split}
\]

Observe that it also induces a birational map
$$
\Phi : \P^{d-\d} \times \P^{\d} \dto \P(V)~.
$$
Since $\Phi$ is birational, one can lift $f_V$ to a product map
$\tilde{f}_V : \P^{d-\d} \times \P^{\d} \dto \P^{d-\d} \times \P^{\d}$
which acts as $\tilde{f}_V=\id\times f_{V_0}$.

Therefore, we conclude the following


\begin{thm}
Suppose $V$ is power associative. Let $k=\dim_\C\left(\red\C[x]\right)$ for any $x\in U$
as described above.
Then for any $p$, we have
$$
\deg_p (f_V^n)  \asymp \left(2^{\min \{ p, k \}}\right)^n
$$
\end{thm}
\mbox{}\hfill\qed

%
%

\begin{ex}
Let $V=M_m(\C)$ be the algebra of $m\times m$ complex matrices. For a matrix $A$ with $m$ distinct
eigenvalues $\mu_1,\ldots,\mu_m$, its characteristic polynomial is also the minimal polynomial.
Thus $\C[A]\cong \oplus_{i=1}^m \C [x] /( x-\mu_i)\cong \C^m$.
Notice that ``having $m$ distinct eigenvalues'' is a generic property in $M_m(\C)$.
Therefore, for the matrix algebra $M_m(\C)$, the number $k$ in the theorem is equal to $m$.

\end{ex}


\section{Generalization to rational maps}

\subsection{Polynomial maps}
Suppose $V$ is power associative and pick any polynomial $P\in \C[T]$.
Then one can look at the map $f_P(v) = P(v)$ on $V$ and compute the
degree growth of $f_P$ on the affine space $V$.

\begin{thm}
Suppose $V$ is power associative, and pick $P\in \C[T]$.
Then there exists $k$ such that for any $p$, we have
$$
\deg_p (f_P^n)  \asymp \left(\deg(P)^{\min \{ p, k \}}\right)^n
$$
\end{thm}
\begin{proof}
By the same trick as in the previous section, we first treat the case that
$V$ is abelian and generated by one element. That is we assume
$$
V \cong \C[T] / \prod_{i=1}^l (T - z_i)^{m_i} \cong \prod_{i=1}^l \C [T_i] /( T_i^{m_i} )~.
$$
Since $f_P$ preserves each factor of the product decomposition above,
it is sufficient to treat the case $V = \C[x]/ (x^m)$.
A point in $V$ can then be written as $\sum_{i=0}^{m-1} \la_i x^i$
and if $P(T) = \sum_j a_j T^j$ we get
\begin{multline*}
P \left( \sum_{i=0}^{m-1} \la_i x^i \right) =
\sum_j 
  a_j  \left( \sum_{i=0}^{m-1} \la_i x^i \right)^j
\\
=P(\lambda_0) + x \left[ \lambda_1 P'(\la_0) \right] + x^2 \left[\la_2 P'(\la_0) + {\frac 1 2} \la_1^2 P''(\la_0) \right]+ \ldots \\
+ x^j \left[\la_j P'(\la_0) + Q_j(\la_0,\la_1, \ldots , \la_{j-1})\right] + \ldots \\
= P(\lambda_0) + \sum_{j=1}^{m-1} x^j \bigl( \la_j P'(\la_0) + Q_j(\la_0,\la_1, \ldots , \la_{j-1})\bigr)
\end{multline*}
In other words, we may write
\begin{align*}
f_V ( \la_0, \la_1 , \ldots , \la_{m-1}) = \bigl(\ & P(\lambda_0) ,  \lambda_1 P'(\la_0), \la_2 P'(\la_0) + \textstyle{\frac 1 2} \la_1^2 P''(\la_0), \\
  & \ldots ,   \lambda_j P'(\la_0) + Q_j (\la_0,\ldots,\la_{j-1}) ,  \\
  & \ldots ,  \la_{m-1} P'(\la_0) + Q_{m-1}(\la_0,\la_1, \ldots , \la_{m-2})\ \bigr)
\end{align*}

Observe that the constant term (i.e., the first coordinate of the function) is just $P(\la_0)$, and for
$1\le j\le m-1$ the coefficient of $x^j$ (i.e., the $j$-th coordinate of the function)
has the following two properties:
\begin{itemize}
\item[(1)] it is a polynomial function of $\la_0,\la_1,\ldots,\la_j$;
\item[(2)] it is an affine function in $\la_j$.
\end{itemize}

If for $0\le i\le m-1$ we let $V^{(i)}=\C^{i+1}$ be the first $i+1$
coordinates of $V$, then by observation (1)
the map $f_P$ induces a selfmap on each $V^{(i)}$. Moreover, for $1\le i\le k-1$ let $\pi_i: V^{(i)}\to V^{(i-1)}$
be the projection, then $f_P|_{V^{(i)}}$ preserve the fibration $\pi_i$ and the map on a generic fiber is
a linear isomorphism of $\P^1$ by (2) above. Moreover, for each $i\ge 1$,
the function on the last coordinate
$(f_P)_i:V^{(i)}\to V^{(i)}/V^{(i-1)}\cong \C$,
as a function of $\la_0,\cdots,\la_{i-1}$, is of degree either
$\deg(P)-1$ (for $i=1$) or $\deg(P)$ (for $i\ge 2$).

When we pass to the $n$-th iterate, we can see that $f_P^n= f_{P^n}$ and $\deg(P^n)=\deg(P)^n$.
We then use Proposition~\ref{lem:deg_growth} repeatedly for $i=1,\cdots,m-1$.
Notice that by (1) above, the number
$\delta_1$ in Proposition~\ref{lem:deg_growth} always equals to one (for each $i$); by the previous paragraph,
the number $\delta_d$ in Proposition~\ref{lem:deg_growth} is asymptotic to $\deg(P)^n$.
This implies that the degree growth for each $p$ is indeed
\[
\deg_p (f_P^n)  \asymp \deg(P)^{pn}.
\]
This is for the special case of $\C[x] \cong \C[T]/ (T^m)$. If $V = \C[x]$, then the same formula is true 
by~\eqref{eq:product} since we have
$\C[x] \cong \prod_i \C [T_i] /( T_i^{m_i} )$. 

Finally let $k$ be the dimension of $\red\C[x]$ for a generic $x\in V$.
By the same argument as in \S\ref{sec:pwassoc}, we know that for generic $x,y$, we have
$\C[x]\cong \C[y]$, and we can further use this fact to make $f_P$ a product map.
Therefore, we can conclude again that we have
\[
\deg_p (f_P^n)  \asymp \left(\deg(P)^{\min \{ p, k \}}\right)^n.
\]
\end{proof}

\subsection{Rational maps}

Next, we claim that for a rational function $\varphi(T)=\frac{Q(T)}{P(T)}\in \C[T]$,
we have the same result for degree growth
for the induced map $f(v):=P(v)^{-1}Q(v)$ on $V$. We will show in a moment that for a generic $v\in V$,
$P(v)$ is invertible, thus $f$ induces a dominant rational map from $V$ to itself.

First, assume $V = \C[x]\cong \C[T]/(T^m)$, $v=\sum_{i=0}^{m-1} \la_i x^i$,
and $P(x) = \sum_i a_i x^i$.
We get
\[
P(v)
= P(\lambda_0) + \sum_{j=1}^{m-1} x^j \bigl( \la_j P'(\la_0) + Q_j(\la_0,\la_1, \ldots , \la_{j-1})\bigr)
\]
If $P(\la_0)\ne 0$, which is the generic case, then
$P(v)-P(\la_0)\in xV$ is nilpotent. Thus $P(v)$ is invertible, and its inverse is given by
\begin{align*}
P(v)^{-1}
&=& \left\{P(\lambda_0)\left(1+\sum_{j=1}^{m-1} x^j
                         \cdot \frac{\la_j P'(\la_0) + Q_j(\la_0,\la_1, \ldots , \la_{j-1})}
                         {P(\la_0)}\right)\right\}^{-1}\\
&=& P(\lambda_0)^{-1}\cdot \sum_{i=0}^{m-1}\left(-\sum_{j=1}^{m-1} x^j
                         \cdot \frac{\la_j P'(\la_0) + Q_j(\la_0,\la_1, \ldots , \la_{j-1})}{P(\la_0)}\right)^i
\end{align*}

Expanding the last line, we get a polynomial expression in $x$.
In order to find the coefficient for $x^j$, we observe that in the expansion, $x^j$ can be formed by
products of terms coming from $x^{j_1}, \ldots,  x^{j_\ell}$ with $j_1+\ldots + j_\ell=j$. If, say $j_1=j$, then all the others satisfy $j_i=0$.
However, since $j$ starts from $1$ in the sum, this
 means we are looking at the term of $x^j$ in the linear term $i=1$,
and the contribution of the coefficient from that product is
$$\la_j \, \frac{-P'(\la_0)}{P(\la_0)^2}+\frac{-Q_j(\la_0,\cdots,\la_{j-1})}{P(\la_0)^2}~.$$
If all $j_i < j$, then the contribution of the coefficient
from that product is a polynomial in $\la_1,\ldots,\la_{j-1}$ for a generic $\la_0$, and
is rational in $\la_0,\la_1,\cdots,\la_{j-1}$. More precisely, the contribution from that
product is of the form
\[
\frac {\widetilde{Q}_j(\la_1,\cdots,\la_{j-1})}{P(\la_0)^\ell},
\]
and since $1\le \ell\le m$,
the degree of this rational function is equivalent to $\deg(P)$  asymptotically.
Therefore, we conclude that the coefficient for $x^j$ is a linear function in $\la_j$
and is a rational function in $\la_0,\cdots,\la_{j}$ of degree asymptotically equivalent to
$\deg(P)$.

Furthermore, if we expand the product
$f_\varphi(v) = P(v)^{-1}Q(v)$ and look at the coefficient of $x^j$,
then the argument in the previous paragraph can also be applied. We have that the constant term
is $Q(\la_0)/P(\la_0)$. For $1\le j\le m-1$, the coefficient of $x^j$
is a rational function in $\la_0,\cdots,\la_j$ of degree asymptotic equivalent to
$\deg(\varphi):=\max\{\deg(P),\deg(Q)\}$,
and for generic $\la_0$ it
is a linear function of $\la_j$. Thus using the same notation as in the polynomial case,
$f$ induces a selfmap on each $V^{(i)}$ and for $2\le i\le m$, $f_P|_{V^{(i)}}$ preserve the
fibration $\pi_i$ and the map on a generic fiber is a linear isomorphism of $\P^1$.
Moreover, we have $f_{\varphi}^n=f_{\varphi^n}$, where $\varphi^n$ means we iterate the single
variable rational map $\varphi(T)$ for $n$ times. Also, notice that
$\deg(\varphi^n)=\deg(\varphi)^n$.
Thus we can use Proposition~\ref{lem:deg_growth} on each $\pi_i$ inductively again. To conclude, we obtain
in this case that $\C[x] \cong \C[T]/ (T^m)$, and we have
$\deg_p (f_\varphi^n)  \asymp \max\{ \deg(Q), \deg(P)\}^{pn}$.

In general,  we can write
$\C[x] \cong \prod_i \C [T_i] /( T_i^{m_i} )$, and for generic $x,y$, we have
$\C[x]\cong \C[y]$, and $f_P$ is birational to a product map.
Therefore, we can conclude again that for $k=\dim_\C\left(\red\C[x]\right)$, where $x$
is generic in $V$, for any $p$, we have
$$
\deg_p (f_\varphi^n)  \asymp \left( \max\{ \deg(Q), \deg(P)\}^{\min \{ p, k \}}\right)^n
$$
This completes the proof of Theorem A.\hfill\qed


\section{Maps of several variables}

In this section, we will assume that $V$ is an abelian (i.e., commutative, associative and unitary)
 $\C$-algebra, and $\dim_{\C}(V)=k$.
Take any dominant rational map
\[
f=(f_0:\cdots:f_d):\C^d\dashrightarrow\C^d,
\]
where each $f_j$ is a rational function.
Interpreting the multiplication and inverse in $f_j$ as multiplication and inverse in $V$
(notice that a generic element in $V$ is invertible),
$f$ also induces a rational map $F:V^d\dashrightarrow V^d$.

As we saw in Section~\ref{sec:square_abelian}, $V$ is Artinian and can be factored as
a product $V=\prod_{i=1}^m V_i$ of local Artinian $\C$-algebras. Let $\frm_i$ be the
maximal ideal of $V_i$, $\pi_i:V_i\to V_i/\frm_i\cong\C$ be the quotient map, and
$F_i:V_i^d\dashrightarrow V_i^d$ be the component of $F$ on $V_i$.
Then one sees that the fibration defined by
$\prod\pi_i:V_i^d\to\C^d$ is preserved by $F_i$, and
the induced map on the base $\C^d$ is exactly $f$. That is, the following diagram of maps
is commutative.
\[
\xymatrix{
V_i^d \ar@{-->}[rr]^{F_i}\ar[d]_{\prod\pi_i} && V_i^d \ar[d]^{\prod\pi_i} \\
\C^d \ar@{-->}[rr]_{f} && \C^d
}
\]
Coming back to $V$, we can in fact describe the fibration more succinctly using
the quotient map $\pi: V\to V/N(V)\cong\C^m$. That is, the fibration is given by
$\prod_{i=1}^k\pi : V^d \to (V/N(V))^d\cong\C^{md}$, and $F$ is preserving the fibration in
the sense that the following diagram is commutative.
\[
\xymatrix{
V^d \ar@{-->}[rr]^{F}\ar[d]_{\prod\pi} && V^d \ar[d]^{\prod\pi} \\
\C^{md} \ar@{-->}[rr]_{\prod f} && \C^{md}
}
\]
In particular, if $N(V)\ne (0)$, then $\prod\pi$ is a fibration with positive dimensional fibers.
In the following, we show that
in the special case of monomial maps, $F$ is moreover a product map.

\subsection{Generalized monomial maps}
Given an $d\times d$ integer matrix
$A=(a_{ij})\in M_d(\Z)$ with $\det(A)\ne 0$, we define the (generalized) monomial map
$F_A:V^d \dashrightarrow V^d$ as
\[
\textstyle{ F_A(x_1,\cdots,x_d)=\left(\prod_j x_j^{a_{1,j}},\cdots ,\prod_j x_j^{a_{d,j}} \right),  }
\]
where $x_j\in V$. We then compactify $V^d\subset \P^{dk}$ and lift the monomial map
as $F_A:\P^{dk}\dashrightarrow \P^{dk}$.

The goal of this section is to prove Theorem B. That is,
we will compute the degree growth of the generalized
monomial map $F_A$. The method we use is similar to Section~\ref{sec:square_abelian} when we
dealt with the squaring map on an abelian algebra.

Now $V$ is finite dimensional as a complex vector space, hence is Artinian as a $\C$-algebra,
hence can be decomposed as
a finite direct product of Artinian local rings $V \simeq \prod_{i=1}^m V_i$.

When $V$ is an Artin local ring with maximal ideal $\frm$, then $\frm$ is nilpotent,
i.e. $\frm^l=0$ for some $l\ge 1$, and $V/\frm\simeq \C$.
As a $\C$-vector space, we can write $V$ as the
direct sum $V\simeq \C\oplus \frm$, where $\frm$ is the maximal ideal of $V$.
Introduce the map
\[
\Phi : \C \oplus \frm \longrightarrow V
\]
sending $(a,h)$ to $a\cdot \exp(h)$, where $\exp(h)=\sum_{i=0}^\infty h^i/i!$.
The function $\exp(h)$ is well defined since the sum is indeed finite, as we explained before.
Moreover, $\Phi$ is a birational map with birational inverse
\[
\Phi^{-1}(a,h)= a\cdot \log(1+a^{-1}h).
\]
Here $\log(1+x)=\sum_{i=1}^\infty (-1)^{i-1}x^i/i$. Notice that the $\Phi^{-1}$ above defines a map $\C \oplus \frm \longrightarrow V$, but using the identification $V\simeq\C \oplus \frm $ (as vector spaces) we can interpret it as a map $V \longrightarrow \C \oplus \frm $. Moreover, one can check that it indeed is the rational inverse of $\Phi$ when we interpret it this way.

For $d$ copies of $V$, we have, as $\C$-vector spaces, $V^d \simeq \C^d \oplus \frm^{\oplus d}$. We use
$\frm^{\oplus d}$ to stress that it is the direct sum of $d$ copies of $\frm$, not the usual
power of ideal $\frm^d$.
Then, the map $F_A$, after conjugating $\Phi$, becomes a product of a monomial map and a linear map, i.e.,
\[
 F_A\circ \Phi =  \Phi\circ ( f_A , T_A),
\]
where $f_A:\C^d\dashrightarrow\C^d$ is the usual monomial map on $\C^d$ induced by $A$, and $T_A$ is the
linear map given by $A$ on the vector space $\frm^{\oplus d}$. 
The linear map $T_A$ has degree $1$ in codimension $0\le j\le \dim(\frm^{\oplus d}) = d(k-1)$, and degree $0$ for $j > d(k-1)$. The degree growth of the usual monomial map $f_A$ is shown
as
\[
\deg_p(f_A) \asymp \left\| \wedge^{p} A^n \right\|,
\]
where $\|\cdot\|$ is any norm (for the proof of this result, see \cite{FW,L}).

Therefore, by \eqref{eq:product}, the degree growth of the map $F_A$ can be described as
\[
\deg_p(F_A^n) \asymp \max_{i+j=p} \{ \deg_i(f_A^n)\deg_j(T_A^n) \} 
\asymp \max_{p+d-dk\le i\le p}
\left\| \wedge^{i} A^n \right\|.
\]
Here, we use the convention that negative exterior product is zero.

In the general case, we have $V \simeq \prod_{i=1}^m V_i$ with each $V_i$ being an Artin local ring.
Hence $F_A:V^d\dashrightarrow V^d$ is a product of maps $F_{A,i}:V_i^d\dashrightarrow V_i^d$, and each
$F_{A,i}$ has a further product structure as explained above.
The map $F_A$ is the again a product of usual monomial maps and a linear map.
The monomial map maps $\C^{md}\to\C^{md}$ and is associated to the matrix
\[
\left(
\begin{matrix}
A & 0 & \cdots & 0 \\
0 & A & \cdots & 0 \\
\vdots & \vdots & \ddots & \vdots \\
0 & 0 & \cdots & A
\end{matrix}
\right).
\]

Let $\diag(A;m)$ denote the above block-diagonal matrix with
$m$ blocks of the matrix $A$ on the diagonal positions. The degree growth behavior of $F_A$ is
therefore as follows.
\[
\deg_p(F_A^n) \asymp \max_{p-d(k-m)\le i\le p} \left\| \wedge^{i} \diag(A;m)^n \right\|.
\]
Notice that the number $m$ is the number of copies of $V_i$'s in the decomposition for $V$.
Another way to write $m$ is as
$m=\dim_\C(\red(V))$, where $\red(V)=V/N(V)$
and $N(V)$ is the nilradical of $V$, i.e., the ideal of all nilpotent elements.
This completes the proof of Theorem B.



\begin{bibdiv}
\begin{biblist}


\bib{AABM}{article}{
   author={Abarenkova, N.},
   author={Angl{\`e}s d'Auriac, J.-Ch.},
   author={Boukraa, S.},
   author={Maillard, J.-M.},
   title={Growth-complexity spectrum of some discrete dynamical systems},
   journal={Phys. D},
   volume={130},
   date={1999},
   number={1-2},
   pages={27--42},
   issn={0167-2789},
}



\bib{AMV}{article}{
   author={Angl{\`e}s d'Auriac, J.-Ch.},
   author={Maillard, J.-M.},
   author={Viallet, C. M.},
   title={On the complexity of some birational transformations},
   journal={J. Phys. A},
   volume={39},
   date={2006},
   number={14},
   pages={3641--3654},
   issn={0305-4470},
}




\bib{AM}{book}{
   author={Atiyah, M. F.},
   author={Macdonald, I. G.},
   title={Introduction to commutative algebra},
   publisher={Addison-Wesley Publishing Co., Reading, Mass.-London-Don
   Mills, Ont.},
   date={1969},
   pages={ix+128},
}





\bib{EB}{article}{
   author={Bedford, Eric},
   title={The dynamical degrees of a mapping},
   conference={
      title={Proceedings of the Workshop Future Directions in Difference
      Equations},
   },
   book={
      series={Colecc. Congr.},
      volume={69},
      publisher={Univ. Vigo, Serv. Publ., Vigo},
   },
   date={2011},
   pages={3--13},
   review={\MR{2905566 (2012m:37080)}},
}



\bib{BK1}{article}{
   author={Bedford, Eric},
   author={Kim, Kyounghee},
   title={On the degree growth of birational mappings in higher dimension},
   journal={J. Geom. Anal.},
   volume={14},
   date={2004},
   number={4},
   pages={567--596},
   issn={1050-6926},
}




\bib{BK2}{article}{
   author={Bedford, Eric},
   author={Kim, Kyounghee},
   title={Degree growth of matrix inversion: birational maps of symmetric,
   cyclic matrices},
   journal={Discrete Contin. Dyn. Syst.},
   volume={21},
   date={2008},
   number={4},
   pages={977--1013},
   issn={1078-0947},
}



\bib{BHM}{article}{
   author={Boukraa, S.},
   author={Hassani, S.},
   author={Maillard, J.-M.},
   title={Noetherian mappings},
   journal={Phys. D},
   volume={185},
   date={2003},
   number={1},
   pages={3--44},
   issn={0167-2789},
}




\bib{CeDe}{article}{
   author={Cerveau, Dominique},
   author={D{\'e}serti, Julie},
   title={It\'eration d'applications rationnelles dans les espaces de
   matrices},
   journal={Conform. Geom. Dyn.},
   volume={15},
   date={2011},
   pages={72--112},
   issn={1088-4173},
}

\bib{NB}{article}{
   author={Dang, Nguyen Bac},
   title={Degrees of iterates of rational maps on normal projective varieties},
   status={preprint}
}


\bib{DF}{article}{
   author={Diller, J.},
   author={Favre, C.},
   title={Dynamics of bimeromorphic maps of surfaces},
   journal={Amer. J. Math.},
   volume={123},
   date={2001},
   number={6},
   pages={1135--1169},
   issn={0002-9327},
}



\bib{DN}{article}{
   author={Dinh, Tien-Cuong},
   author={Nguy{\^e}n, Vi{\^e}t-Anh},
   title={Comparison of dynamical degrees for semi-conjugate meromorphic
   maps},
   journal={Comment. Math. Helv.},
   volume={86},
   date={2011},
   number={4},
   pages={817--840},
   issn={0010-2571},
}



\bib{DNT}{article}{
   author={Dinh, Tien-Cuong},
   author={Nguy{\^e}n, Vi{\^e}t-Anh},
   author={Truong, Tuyen Trung},
   title={On the dynamical degrees of meromorphic maps preserving a
   fibration},
   journal={Commun. Contemp. Math.},
   volume={14},
   date={2012},
   number={6},
   pages={1250042, 18},
   issn={0219-1997},
}




\bib{DS}{article}{
   author={Dinh, Tien-Cuong},
   author={Sibony, Nessim},
   title={Une borne sup\'erieure pour l'entropie topologique d'une
   application rationnelle},
   journal={Ann. of Math. (2)},
   volume={161},
   date={2005},
   number={3},
   pages={1637--1644},
   issn={0003-486X},
}



\bib{FJ}{article}{
   author={Favre, Charles},
   author={Jonsson, Mattias},
   title={Dynamical compactifications of ${\bf C}^2$},
   journal={Ann. of Math. (2)},
   volume={173},
   date={2011},
   number={1},
   pages={211--248},
   issn={0003-486X},
}



\bib{FW}{article}{
   author={Favre, Charles},
   author={Wulcan, Elizabeth},
   title={Degree growth of monomial maps and McMullen's polytope algebra},
   journal={Indiana Univ. Math. J.},
   volume={61},
   date={2012},
   number={2},
   pages={493--524},
   issn={0022-2518},
}



\bib{L}{article}{
   author={Lin, Jan-Li},
   title={Pulling back cohomology classes and dynamical degrees of monomial
   maps},
   journal={Bull. Soc. Math. France},
   volume={140},
   date={2012},
   number={4},
   pages={533--549 (2013)},
   issn={0037-9484},
}



\bib{Ng}{article}{
   author={Nguy{\^e}n, Vi{\^e}t-Anh},
   title={Algebraic degrees for iterates of meromorphic self-maps of ${\mathbb
   P}^k$},
   journal={Publ. Mat.},
   volume={50},
   date={2006},
   number={2},
   pages={457--473},
   issn={0214-1493},
}






\bib{Og}{article}{
   author={Oguiso, Keiji},
   title={A remark on dynamical degrees of automorphisms of hyperk\"ahler
   manifolds},
   journal={Manuscripta Math.},
   volume={130},
   date={2009},
   number={1},
   pages={101--111},
   issn={0025-2611},
}




%


\bib{PiRu1}{article}{
   author={Pirio, Luc},
   author={Russo, Francesco},
   title={Quadro-quadric Cremona maps and varieties 3-connected by cubics:  semi-simple part
and radical},
   journal={International Journal of Mathematics},
   volume={24},
   number={13},
   date={2013},
   pages={1350105},
}


\bib{PiRu2}{article}{
   author={Pirio, Luc},
   author={Russo, Francesco},
   title={Quadro-quadric Cremona transformations in low dimensions via the JC-correspondence},
   journal={Ann. Inst. Fourier},
   volume={64},
   date={2014},
   pages={71-111},
}

\bib{PiRu3}{article}{
   author={Pirio, Luc},
   author={Russo, Francesco},
   title={The XJC-correspondence},
   journal={Journal f\"ur die reine und angewandte Mathe-
matik},
   volume={716},
   date={2016},
   pages={229--250},
}


\bib{MR1488341}{article}{
   author={Russakovskii, Alexander},
   author={Shiffman, Bernard},
   title={Value distribution for sequences of rational mappings and complex
   dynamics},
   journal={Indiana Univ. Math. J.},
   volume={46},
   date={1997},
   number={3},
   pages={897--932},
   issn={0022-2518},
}


\bib{Tru}{article}{
   author={Truong, Tuyen Trung},
   title={(Relative) dynamical degrees of rational maps over an algebraic closed field},
   eprint={arXiv:1501.01523 [math.AG]}
}



\bib{Usn}{article}{
   author={Usnich, Alexander},
   title={A discrete dynamical system acting on pairs of matrices},
   journal={Dokl. Nats. Akad. Nauk Belarusi},
   volume={53},
   date={2009},
   number={3},
   pages={21--24, 124},
   issn={0002-354X},
}



\end{biblist}
\end{bibdiv}
\end{document}